\documentclass[12pt,
reqno]{amsart}

\usepackage[cp1251]{inputenc}
\usepackage{amsmath,amsxtra,amssymb,eufrak}
\usepackage[all]{xy}
\usepackage[active]{srcltx} 
\usepackage{mathrsfs}
\usepackage[english]{babel}

\newtheorem{thm}{Theorem}[section]
\newtheorem{lm}[thm]{Lemma}
\newtheorem{co}[thm]{Corollary}
\newtheorem{pr}[thm]{Proposition}

\theoremstyle{definition}

\newtheorem{rem}[thm]{Remark}

\numberwithin{equation}{section}

\newcommand{\ptn}{\mathbin{\widehat{\otimes}}}

\newcommand{\CC}{\mathbb{C}}

\newcommand{\R}{\mathbb{R}}
\newcommand{\Z}{\mathbb{Z}}
\newcommand{\N}{\mathbb{N}}

\newcommand*{\DD}{\mathbb D}

\newcommand{\B}{{\mathrm{B}}}

\newcommand*{\rT}{\mathrm T}

\newcommand{\cO}{\mathcal{O}}
\newcommand*{\cR}{\mathcal R}

\newcommand*{\fg}{\mathfrak{g}}

\newcommand*{\wh}{\widehat}
\newcommand*{\wt}{\widetilde}

\renewcommand{\le}{\leqslant}
\renewcommand{\ge}{\geqslant}

\let \al         =\alpha
\let \be         =\beta
\let \ga         =\gamma

\let \te         =\theta        
\let \io         =\iota

\let \la         =\lambda

\let \om         =\omega

\let \Ga         =\Gamma

\let \Om         =\Omega

\let \phi         =\varphi

\title{On two versions of holomorphic quantum plane}
\author{O. Yu. Aristov}
\address{Institute for Advanced Study in Mathematics of Harbin Institute of Technology, Harbin 150001, China;
\newline\indent
Suzhou Research Institute of Harbin Institute of Technology, Suzhou 215104, China}
\email{aristovoyu@inbox.ru}

\begin{document}
\begin{abstract}
We find power series descriptions of two versions of holomorphic quantum plane, the Arens--Michael envelope and the envelope with respect to the class of Banach PI algebras, in the case of non-unitary parameter.
\end{abstract}

 \maketitle

\section*{Introduction}
We consider two completions of the universal complex associative algebra with generators $x$ and $y$ satisfying the relation $xy=qyx$ with a complex parameter $q$ (Manin's quantum plane). The first completion is the envelope with respect to the class of all Banach algebras (the Arens--Michael envelope) and the second is  the envelope with respect to the class of Banach algebras satisfying a polynomial identity (PI algebras). We denote them by $\cO(\CC^2_q)$ and $\cO(\CC^2_q)^{\mathrm{PI}}$, respectively. Both algebras deserve the name 'holomorphic quantum planes'.

It seems that the first to study analytical versions of quantum affine spaces (in particular, quantum planes) was Pirkovskii \cite{Pir_qfree} in 2008. Interest in this topic has been revived lately; see \cite{ArNew,Do24a,Do24,Do24b}. On the other hand, the class $\mathsf{PI}$ of Banach PI algebras have been studied but not very actively; see \cite{Kr87,Mu94}. Recently, however, it was discovered that this area is connected to non-commutative geometry; see the papers of the author \cite{ArOld,Ar22,ArPiLie,Ar_smash,ArNew,ArNew2}. Specifically,  envelopes with respect to $\mathsf{PI}$ were introduced in \cite{Ar_smash}. For classical algebras arising in non-commutative algebraic geometry, such envelopes can be, along with Arens--Michael envelopes, treated as objects of study in non-commutative complex-analytic geometry. Furthermore, envelopes with respect to $\mathsf{PI}$ are  often easier to work with  and have a simpler structure. This feature is demonstrated in this paper by using quantum planes as an example.

Our main aim is to find power series representations of $\cO(\CC^2_q)$ and $\cO(\CC^2_q)^{\mathrm{PI}}$. Note that the following description of elements  of $\cO(\CC^2_q)$ as series in powers of $x$ and $y$ is given in \cite[Corollary 5.14]{Pir_qfree}:
$$
\cO(\CC^2_q)=\Bigl\{a=\sum_{i,j=0}^\infty
\al_{ij} y^i x^j\!:\,
 \sum_{i,j=0}^\infty |\al_{ij}|\, r^{i+j}<\infty \;\;\forall r>0\Bigr\}.
$$
(Here $|q|\le 1$. The case when $|q|>1$  can be easily reduced to this one by transposing $x$ and $y$.) We are not able to add something new to this picture for $|q|=1$. But in the case when $|q|<1$, we show that both $\cO(\CC^2_q)$ and $\cO(\CC^2_q)^{\mathrm{PI}}$ can be written in a more structural form with the use also of powers of the product $u=xy$. (Note that the difference between the cases of $|q|=1$ and $|q|\ne 1$ naturally arises in the study of other holomorphic quantum algebras; see \cite{ArDJ}.) Specifically, as locally convex spaces,
$$
\cO(\CC^2_q)\cong\cO(\Om)\ptn\mathfrak{B}_{|q|^{1/2}}\quad\text{and}\quad
\cO(\CC^2_q)^{\mathrm{PI}}\cong\cO(\Om)\ptn\CC[[u]],
$$
where $\ptn$ stands for the complete projective tensor product of locally convex spaces, $\cO(\Om)$ for the algebra of holomorphic functions on the Gelfand spectrum, $\CC[[u]]$ for the algebra of all formal power series in $u$, and $\mathfrak{B}_{|q|^{1/2}}$ for an algebra of formal power series with a certain restriction on growth.

Our approach is motivated by a study of a $C^\infty$-version of the quantum plane in \cite[\S\,4]{ArNew}. Each of the concomitant Banach algebra in the  $C^\infty$-version automatically satisfies to a polynomial identity, and so it is not surprising that we can use the method developed in \cite{ArNew} for the algebra $\cO(\CC^2_q)^{\mathrm{PI}}$. (Note that in the PI-variant the only restriction we require is that $q\ne 1$.) Indeed, we demonstrate that the $C^\infty$-argument can be applied in this situation. However, it is remarkable that it can also be modified for the more involved case of $\cO(\CC^2_q)$.

 \subsection*{Acknowledgments} 
A part of this work was done during a visit to the HSE University (Moscow) in the winter  of 2025. I wish to thank this university 
 for the hospitality.

\section{Preliminaries and statement of results}

In this text, we consider only unital associative algebras over the field $\CC$, except for Remark~\ref{Lieni}, which mentions Lie algebras.

Recall that a topological algebra is called an \emph{Arens--Michael algebra} if it is complete and isomorphic to a projective limit of Banach algebras. Denote by $\mathsf{PI}$ the class of Banach algebras satisfying a polynomial identity.  (We call them Banach PI algebras.) Following \cite[Definition 5.4]{Ar_smash}  we say that a topological algebra is a \emph{locally in $\mathsf{PI}$} if it is isomorphic to a projective limit of algebras contained in $\mathsf{PI}$. The term \emph{locally BPI algebra} is also used; see  \cite[Definition 1.1]{ArNew} for a more general context.

Recall that an \emph{Arens--Michael envelope of an associative algebra $A$} is a pair $(\wh A,\io)$, where $\wh A$ is an Arens--Michael algebra and $\io$ is a homomorphism $A \to \wh A$ such that for every Banach algebra (equivalently,  Arens--Michael algebra) $B$ and every homomorphism $\phi\!: A \to B$ there is a unique continuous homomorphism
$\widehat\phi\!:\wh A\to B$ making the diagram
\begin{equation*}
  \xymatrix{
A \ar[r]^{\io}\ar[rd]_{\phi}&\wh A\ar@{-->}[d]^{\widehat\phi}\\
 &B\\
 }
\end{equation*}
commutative.

In \cite{ArNew} this definition  has been generalised to an arbitrary class of Banach algebras. In particular, we can take $\mathsf{PI}$; see \cite{Ar_smash}. More specifically, an \emph{envelope of an algebra $A$ with respect to the class $\mathsf{PI}$}  is a pair $(\wh A^{\,\mathsf{PI}},\io)$, where is $\wh A^{\,\mathsf{PI}}$ is locally in  $\mathsf{PI}$ and $\io$ is a homomorphism $A \to \wh A^{\,\mathsf{PI}}$ that satisfies the same universal property but in the class of Banach PI algebras (equivalently, locally BPI algebras); see \cite[Definition 5.4]{Ar_smash} and  \cite[Definition 1.2]{ArNew}. Note that  Arens--Michael envelopes and envelopes with respect to $\mathsf{PI}$ always exist; see \cite[Exercise V.2.24]{He93}  and \cite[Proposition 5.7]{Ar_smash}, respectively; cf. \cite[Proposition 1.4]{ArNew}.

For $q\in\CC$ denote  by $\mathcal{R}(\CC^2_q)$ the universal complex associative algebra generated by $x$ and $y$ subject to relation $xy=qyx$.
First we write $\cR(\CC^2_q)$ in a form  that will be convenient for later use. Put
\begin{align}
\label{cROm}\mathcal{R}(\Om)\!&:=\{(f,g)\in \CC[t]\times \CC[t]\!:\,f(0)=g(0)\};\\
\label{cOOm}\cO(\Om)\!&:=\{(f,g)\in \cO(\CC)\times \cO(\CC)\!:\,f(0)=g(0)\}.
\end{align}
Note that $\mathcal{R}(\Om)$  and $\cO(\Om)$ can be identified with the quotients of $\CC[x,y]$ and $\cO(\CC^2)$, respectively, by the ideals generated by $xy$. (In the second case the ideal is automatically closed.) So $\Om$ can be identified with the Gelfand spectrum (the set of one-dimensional representations) of $\mathcal{R}(\Om)$ and similarly for $\cO(\Om)$.

Denote the pairs $(t, 0)$ and $(0, t)$ in $\mathcal{R}(\Om)$ by $X$ and $Y$ and put $u=xy$. Then $\cR(\CC^2_q)$ can be identified with $\cR(\Om)\otimes \CC[u]$ via the linear isomorphism
\begin{equation}\label{Riso}
 x^i u^j\mapsto X^i\otimes u^j,\quad y^i u^j\mapsto Y^i\otimes u^j,\quad u^j\mapsto 1\otimes u^j.
\end{equation}
So we can assume that $\cR(\CC^2_q)$ coincides with $\mathcal{R}(\Om)\otimes \CC[u]$ endowed with the multiplication determined by the relations $XY = qYX$ and $XY=u$.

\subsection*{Statement of main results}

For $s\in(0,1)$ denote by $\mathfrak{B}_s$ the universal  Arens--Michael algebra  generated topologically by a single element $u$ satisfying the condition
\begin{equation}\label{xyest}
\|u^n\|^{1/n}=O(s^n)\quad \text{as $n\to\infty$}
\end{equation}
for rach continuous submultiplicative seminorm $\|\cdot\|$.
(The existence of such an algebra is proved in Corollary~\ref{PexunCom} below.)  The importance of  $\mathfrak{B}_s$ for our problem stems from the fact that, in a  Banach algebra, the relation $xy=qyx$ with $|q|<1$ implies that~\eqref{xyest} holds with $u=xy$ and $s=|q|^{1/2}$; see Lemma~\ref{sile}. The algebra $\CC[[z]]$, consisting of all formal power series, is also universal, now with respect to the condition that $z$ is nilpotent; see Lemma~\ref{forseup}.

Using the linear isomorphism $\cR(\CC^2_q)\cong \cR(\Om)\otimes \CC[u]$ described above, we can treat $\cR(\CC^2_q)$ as a vector subspace of both $\cO(\Om)\ptn\mathfrak{B}_{|q|^{1/2}}$ and $\cO(\Om)\ptn\CC[[u]]$. Consider the corresponding embeddings,
\begin{equation}\label{iodef}
\io_1\!:\cR(\CC^2_q)\to \cO(\Om)\ptn\mathfrak{B}_{|q|^{1/2}}\quad\text{and}\quad \io_2\!:\cR(\CC^2_q)\to \cO(\Om)\ptn\CC[[u]].
\end{equation}

The following two theorems are our main results.

\begin{thm}\label{AMqupl}
Let $q\in\CC\setminus\{0\}$ and $|q|< 1$.

\emph{(A)}~The multiplication in $\mathcal{R}(\Om)\otimes \CC[u]$ can be extended to a continuous operation on $\cO(\Om)\ptn\mathfrak{B}_{|q|^{1/2}}$ that turns it into an Arens--Michael algebra.

\emph{(B)}~Taking $\cO(\Om)\ptn\mathfrak{B}_{|q|^{1/2}}$  with this multiplication, the embedding
$$
\io_1\!:\mathcal{R}(\CC^2_q)\to \cO(\Om)\ptn\mathfrak{B}_{|q|^{1/2}}
$$
is an Arens--Michael enveloping homomorphism, i.e., $\cO(\CC^2_q)\cong\cO(\Om)\ptn\mathfrak{B}_{|q|^{1/2}}$.
\end{thm}
In fact,
\begin{equation}\label{Bq12}
\mathfrak{B}_{|q|^{1/2}}=\Bigl\{a=\sum_{n=0}^\infty  \al_n z^n\! :
\|a\|_{r,\om}\!:=\sum_{n=0}^\infty |\al_n|\,r^n |q|^{n^2/2}<\infty
\;\;\forall  r\in(0,\infty)\Bigr\};
\end{equation}
see \eqref{faBrdef}.

\begin{thm}\label{AMquplPI}
\emph{(cf. \cite[Theorem 4.3]{ArNew})}
Let $q\in\CC\setminus\{0\}$.

\emph{(A)}~The multiplication in $\mathcal{R}(\Om)\otimes \CC[u]$ can be extended to a continuous operation on $\cO(\Om)\ptn\CC[[u]]$ that turns it into a locally BPI algebra.

\emph{(B)}~If, in addition, $q\ne 1$, then, taking $\cO(\Om)\ptn\CC[[u]]$  with this multiplication, the embedding
$$
\io_2\!:\mathcal{R}(\CC^2_q)\to \cO(\Om)\ptn\CC[[u]]
$$
is an enveloping homomorphism with respect to $\mathsf{PI}$, i.e., $\cO(\CC^2_q)^{\mathrm{PI}}\cong\cO(\Om)\ptn\CC[[u]]$.
\end{thm}

\begin{rem}\label{Lieni}
Let $\fg$ be a finite-dimensional nilpotent complex Lie algebra and $U(\fg)$ the corresponding universal enveloping algebra. Then, as a locally convex space,
$$
\widehat{U}(\fg)\cong \mathfrak{A}_{i_1}\ptn\cdots\ptn\mathfrak{A}_{i_p}\ptn\cO(\CC^k)
$$
for some $i_1,\ldots,i_p$ and $k$. Here $\mathfrak{A}_{i_1},\ldots,\mathfrak{A}_{i_p}$ are certain power series algebras of the form given in~\eqref{faComdef} below; for details see \cite[Theorem~2.5]{ArSLA}. Other forms of this isomorphism can be found in \cite[Theorem~1.1]{ArAMN},  \cite[Theorem~4.3]{AHHFG} and \cite[Theorem~6.4]{Ar_smash}).  (In the case when $\fg$ is solvable, there are similar but slightly more complicated formulas.)
Furthermore, it follow from \cite[Theorem~6.6]{Ar_smash} that
$$
\widehat{U}(\fg)^{\mathrm{PI}}\cong \CC[[x_1]]\ptn\cdots\ptn\CC[[x_p]]\ptn\cO(\CC^k).
$$
The latter algebra is actually the algebra of `formally-radical functions' considered by Dosi in~\cite{Do10B}.

Thus the degeneracy effect, when some elements in an envelope of a non-commutative algebra generate spaces of power series that are larger than the space of entire functions, is not only a characteristic of quantum planes.
\end{rem}

\section{The Arens--Michael envelope}

In the case when $|q|<1$, our description of $\cO(\CC^2_q)$, the Arens--Michael envelope of  $\mathcal{R}(\CC^2_q)$,  is based on the following simple lemma.
\begin{lm}\label{sile}
Let $X$ and $Y$ be elements of a Banach algebra such that $XY=qYX$ for some
$|q|<1$. Then the growth condition in~\eqref{xyest} holds with $s=|q|^{1/2}$ and  $u=XY$ (or $u=YX$).
\end{lm}
\begin{proof}
It is not hard to see by induction that $(XY)^n=q^{n(n+1)/2}Y^{n}X^n$ for every $n\in\N$. Therefore
$\|(XY)^n\|^{1/n} \le q^{(n+1)/2}\|Y\|\,\|X\|$  and hence~\eqref{xyest} holds.

The case when $u=YX$ is similar.
\end{proof}

First, we demonstrate the existence of a universal algebra in a more general situation by describing its explicit form.
Let $\om=(\om_n;\,n\in\Z_+)$ be a submultiplicative weight, i.e., $\om_n\ge 0$ and $\om_{n+m}\le \om_n\om_m$ for all $n$ and $m$. Consider the power series space
\begin{equation}
 \label{faComdef}
\mathfrak{C}_\om\!:=\Bigl\{a=\sum_{n=0}^\infty  \al_n z^n\! :
\|a\|_{r,\om}\!:=\sum_{n=0}^\infty |\al_n|\,r^n\,\om_n<\infty
\;\;\forall r\in (0,\infty)\Bigr\}
\end{equation}
and endow it with the topology determined by the family $(\|\cdot\|_{r,\om};\,r\in (0,\infty))$.
The submultiplicativity of $\om$ implies  that $\|a_1a_2\|_{r,\om}\le \|a_1\|_{r,\om} \|a_2\|_{r,\om}$ for every $a_1,a_2 \in\mathfrak{C}_\om$. Also, being a K\"{o}the sequence space, $\mathfrak{C}_\om$ is complete. Thus it is an Arens--Michael algebra with respect to the multiplication extended from $\CC[z]$.

In \cite{Gr71,Gr73} Grabiner considered a  power  series  calculus  for a given quasi-nilpotent operator. We need a simple modification, where we take not a single operator but a class of operators or elements of Banach algebras given by a restriction on the growth of powers. We
formulate the existence of calculus as a universal property for $\mathfrak{C}_\om$.

\begin{pr}\label{PexunBr}
Let $\om=(\om_n)$ be a submultiplicative weight and $b$ an element of a Banach algebra~$B$. Suppose that $\|b^n\|^{1/n}=O(\om_n^{1/n})$  as $n\to\infty$. Then there is a unique continuous unital homomorphism $\psi\!:\mathfrak{C}_\om\to B$ that maps $z$ to $b$.
\end{pr}
\begin{proof}
Take $r>0$ such that $\|b^n\|^{1/n}\le r\,\om_n^{1/n}$ for every $n$.
Note that $\psi$ is obviously defined on polynomials by the formula $\sum \al_n z^n\mapsto \sum \al_n b^n$.
Also, if $a=\sum_{n=0}^N \al_n z^n$, then
$$
\|\psi(a)\|\le \sum_{n=0}^N |\al_n| \|b^n\|\le \sum_{n=0}^N |\al_n|\,r^n\,\om_n=\|a\|_{r,\om}.
$$
Thus $\psi$ is continuous and hence extends uniquely to $\mathfrak{C}_\om$.
\end{proof}

Put now $\om_n\!:=s^{n^2}$, where  $s\in(0,1)$. Since $(m+n)^2\ge m^2+n^2$ and $s<1$, we have that
$$
s^{(m+n)^2}\le s^{m^2+n^2}
$$
for every $m,n\in\Z_+$, i.e., $(\om_n)$ is submultiplicative. In this case, we use the notations $\mathfrak{B}_s$ for $\mathfrak{C}_\om$. In detail,
\begin{equation}
 \label{faBrdef}
\mathfrak{B}_s\!:=\Bigl\{a=\sum_{n=0}^\infty  \al_n z^n\! :\,
\sum_{n=0}^\infty |\al_n|\,r^n s^{n^2}<\infty
\;\;\forall  r\in(0,\infty)\Bigr\}.
\end{equation}

We immediately obtain the following corollary of Proposition~\ref{PexunBr}.

\begin{co}\label{PexunCom}
Let $s\in(0,1)$ and $b$ an element of a Banach algebra~$B$. Suppose that
$\|b^n\|^{1/n}=O(s^n)$  as $n\to\infty$. Then there is a unique continuous unital homomorphism
$\psi\!:\mathfrak{B}_s\to B$ that maps $z$ to $b$.
\end{co}

The following extension to Arens--Michael algebras is straightforward.

\begin{co}\label{PexunCom2}
Let $s\in(0,1)$ and $b$ an element of an Arens--Michael algebra~$B$. Suppose that
$\|b^n\|^{1/n}=O(s^n)$  as $n\to\infty$ for every $\|\cdot\|$ in some system of submultiplicative seminorms that determines the topology on~$B$. Then there is a unique continuous unital homomorphism
$\psi\!:\mathfrak{B}_s\to B$ that maps $z$ to $b$.
\end{co}

Next we need two families of infinite-dimensional representations. Consider the standard Banach sequence spaces $c_0$ and $\ell_1$ and denote by $\B(c_0)$ and $\B(\ell_1)$ the Banach algebras of bounded operators on $c_0$ and $\ell_1$, respectively. We use the following notations both for $c_0$ and $\ell_1$:
\begin{itemize}
\item $E$ denotes the operator  of left shift;

\item $F$ denotes the operator  of right shift;

\item $D$ denotes the diagonal operator with the entries $1,q,q^2,\ldots$
\end{itemize}

It is easy to see that $ED=qDE$ and $DF=qFD$ and so the formulas
$$
\pi_{\la}\!:x\mapsto E ,\,y\mapsto \la D\quad\text{and}\quad \pi'_{\mu}\!:x\mapsto \mu D,\,y\mapsto F \qquad(\la,\mu\in\CC)
$$
define bounded representations of $\mathcal{R}(\CC^2_q)$ on $c_0$ and $\ell_1$, respectively. It is convenient to consider them as homomorphisms from $\mathcal{R}(\CC^2_q)$ to $\B(c_0)$ and $\B(\ell_1)$. Treating $\la$ and $\mu$ as variables, we also obtain homomorphisms
$$
\wt\pi\!:\mathcal{R}(\CC^2_q)\to \cO(\CC,\B(c_0))\quad\text{and}\quad\wt\pi'\!:\mathcal{R}(\CC^2_q)\to \cO(\CC,\B(\ell_1))
$$
to the algebras of operator-valued entire functions.

Since the pairs $(E,\la D)$ and $(\mu D,F)$ satisfy the relation Lemma~\ref{sile}, we have  that
$$
\|(\la ED)^n\|^{1/n}=O(|q|^{n/2})\quad\text{and}\quad\|(\mu DF)^n\|^{1/n}=O(|q|^{n/2}) \quad \text{as $n\to\infty$}\quad(\forall\,\la,\mu).
$$
Moreover, the same estimates hold for the systems of standard max-seminorms on $\cO(\CC,\B(c_0))$ and $\cO(\CC,\B(\ell_1))$.


It follows from Corollary~\ref{PexunCom2} that the homomorphisms $u\mapsto \la ED$ and $u\mapsto \mu DF$   extend from $\CC[u]$  to $\mathfrak{B}_{|q|^{1/2}}$. On the other hand, the homomorphisms 
$$
\CC[x]\to\cO(\CC,\B(c_0))\quad\text{and}\quad\CC[y]\to\cO(\CC,\B(\ell_1))
$$  
obviously extend to $\cO(\CC)$. Thus  the restrictions of $\wt\pi$ and $\wt\pi'$ to $\CC[x]\otimes\CC[u]$ and $\CC[y]\otimes\CC[u]$, respectively,  have extensions to continuous linear maps  
$$
\cO(\CC)\ptn\mathfrak{B}_{|q|^{1/2}}\to\cO(\CC,\B(c_0))\quad\text{and}\quad \cO(\CC)\ptn\mathfrak{B}_{|q|^{1/2}}\to\cO(\CC,\B(\ell_1)).
$$

We can identify $\mathcal{R}(\CC^2_q)$ with $\mathcal{R}(\Om)\otimes \CC[u]$ and treat $\mathcal{R}(\Om)$ as a subalgebra of $\CC[t]^2$ (see \eqref{cROm}), $\cO(\Om)$ as a subalgebra of $\cO(\CC)^2$ (see \eqref{cOOm}), and  $\cO(\Om)\ptn\mathfrak{B}_{|q|^{1/2}}$ as a subalgebra of $(\cO(\CC)\ptn\mathfrak{B}_{|q|^{1/2}})^2$ (using the fact that $\mathfrak{B}_{|q|^{1/2}}$ is nuclear). So we get a map
\begin{equation}\label{rhodef}
\rho\!:\cO(\Om)\ptn\mathfrak{B}_{|q|^{1/2}}\to \cO(\CC,\B(c_0))\times \cO(\CC,\B(\ell_1)).
\end{equation}

We want to prove that $\rho$ is topologically injective. For this we need its corestriction.
Take the first row in the matrix representing elements of $\B(c_0)$ and the first column in the matrices representing elements of $\B(\ell_1)$. The  row case gives the Banach space  dual to $c_0$, i.e., $\ell_1$, and the column case also gives $\ell_1$.

So we obtain a map
\begin{equation}\label{etadef}
\eta\!:\cO(\Om)\ptn\mathfrak{B}_{|q|^{1/2}}\to \cO(\CC,\ell_1)\times\cO(\CC,\ell_1)
\end{equation}
in which we denote the first and the second multiples by $\eta_1$ and $\eta_2$, respectively. Next we describe these maps in detail. 

Put $\Phi_x\!:\CC[x]\to \mathcal{R}(\CC^2_q)\!:x\to X$ and $\Phi_y\!:\CC[y]\to \mathcal{R}(\CC^2_q)\!:y\to Y$; cf. \eqref{Riso}. Then we can write every element of $\mathcal{R}(\CC^2_q)$ as
\begin{equation}\label{RPhixy}
a=\sum_{n\ge 0}(\Phi_x(f_n)+\Phi_y(g_n))u^n,
\end{equation}
where $f_n\in\CC[x]$ and $g_n\in\CC[y]$ with $f_n(0)=g_n(0)$. 

By a standard result, we can identify $\cO(\CC,\ell_1)$ with $\cO(\CC)\ptn \ell_1$ (see, e.g., \cite[Chapter\,II, p.\,114, Theorem 4.14]{He89}) and moreover with the vector-valued sequence space $\ell_1[\cO(\CC)]$. (In what follows we enumerate vectors of bases in $c_0$ and $\ell_1$ by non-negative integers.) 
For $\bar h=(h_0,h_1,\ldots)\in\ell_1[\cO(\CC)]$ and $n\in\Z_+$ put
$$
(W_n(\bar h)(z)\!:=h_n(z)z^{n} q^{n(n+1)/2}+\sum_{k=0}^{n} \frac{h_k^{(n-k)}(0)}{(n-k)!}z^k q^{k(k+1)/2}\quad(z\in\CC).
$$
(Here Lagrange's notation for derivatives is used.)

First, we write $\eta_1$ and $\eta_2$ in terms of operators $W_n$.

\begin{lm}\label{eta12}
For every $a\in\mathcal{R}(\CC^2_q)$ given in the form~\eqref{RPhixy} the equalities
\begin{align*}
(\eta_1(a))(\la)& =(W_0(\bar g))(\la),\ldots, (W_n(\bar g))(\la),\ldots),\\
(\eta_2(a))(\mu)& =(W_0(\bar f))(\mu),\ldots, (W_n(\bar f))(\mu),\ldots)^T
\end{align*}
hold.  (Here $T$ stands for the transpose matrix.)
\end{lm}
\begin{proof}
Note that $\pi_\la(u)$ is the operator of weighted left shift with the weight sequence $(\la q, \la q^2, \ldots)$. It follows that the only non-zero entry in the matrix of $\pi_\la(u^n)$ in the first row is $\la^n q^{n(n+1)/2}$ at the $n$th place. On the other hand, we have that
$$
(\eta_1(\Phi_x(f)))(\la)=(f(0),\ldots, {f^{(n)}(0)}/{n!},\ldots)\quad\text{and}\quad (\eta_1(\Phi_y(g)))(\la)=(g(\la),0,0,\ldots) 
$$
for every $f\in\CC[x]$ and $g\in\CC[y]$. Since $f_k(0)=g_k(0)$ for every $k$, the first equality in the lemma follows from~\eqref{RPhixy} and the formula of matrix multiplication.

Similarly, $\pi'_\mu(u)$ is the operator of weighted right shift with the weight sequence $(\mu q, \mu q^2, \ldots)$. The only non-zero entry of $\pi'_\mu(u^n)$ in the first row is $\mu^n q^{n(n+1)/2}$ at the $n$th place, and
$$
(\eta_2(\Phi_x(f)))(\mu)=(f(\mu),0,0,\ldots)^T\quad\text{and}\quad  (\eta_2(\Phi_y(g)))(\mu)=(g(0),\ldots, {g^{(n)}(0)}/{n!},\ldots)^T.
$$
So we get the second equality in the same way as the first.
\end{proof}

We use the standard family of norms on $\cO(\CC)$: $\|f\|_{\rho}\!:= \sup\{|f(z)|\!:\, |z|\le \rho\}$ ($\rho>0$). The following estimate is technical.

\begin{lm}\label{hnset}
Put $W_n= W_n(\bar h)$ and $\|\cdot\|=\|\cdot\|_{\rho}$ for fixed $\bar h$ and $\rho$. Then 
\begin{equation}\label{fhnset}
\|h_n\|\,\rho^n |q|^{n(n+1)/2}\le \frac 32\, \|W_{n} \|+ \left(\frac 32 \right)^2\,\left[ \sum_{k=0}^{n-1} \left(\frac 52 \right)^{k}\frac{\|W_{n-1-k} \|}{\rho^{k+1}}\right]
\end{equation}
for every $n\in\Z_+$.
\end{lm}
\begin{proof}
We proceed by induction.
First we have from $|h(0)|\le  \|h+h(0)\|/2$ that
\begin{equation}\label{32est}
\|h\|\le \|h+h(0)\|+|h(0)|\le \frac 32\, \|h+h(0)\|
\end{equation}
for each $h\in\cO(\CC)$. In particular, 
$$
\|h_0\|\le \frac 32\, \|h_0+h_0(0)\|=\frac 32\, \|W_0\|,
$$
i.e., \eqref{fhnset} holds when $n=0$.

Now assume that \eqref{fhnset} holds  for all non-negative integers less than $n$.
By the definition of $W_{n} $,
$$
h_n(z)+h_n(0)=z^{-n} q^{-n(n+1)/2}\left(W_{n} - \sum_{k=0}^{n-1} \frac{h_k^{(n-k)}(0)}{(n-k)!}z^k q^{k(k+1)/2)}\right).
$$
It follows from this equality and \eqref{32est} that
$$
\|h_n\|\le
  \frac 32\, \|h_n+h_n(0)\| \le
   \frac 32\,\rho^{-n} |q|^{-n(n+1)/2}\left(\|W_{n} \|+ \sum_{k=0}^{n-1} \frac{|h_k^{(n-k)}(0)|}{(n-k)!}\rho^k |q|^{k(k+1)/2}\right).
$$

The Cauchy inequalities states that 
$$\frac{|h^{(n-k)}(0)|}{(n-k)!} \le \frac{\|h\|}{\rho^{n-k}}$$ for every $h\in\cO(\CC)$. Therefore by the  induction hypothesis,
\begin{multline*}
\|h_n\|\rho^{n} |q|^{n(n+1)/2}\le\frac 32\left(\|W_{n} \|+ \sum_{k=0}^{n-1} \|h_k\|\,\rho^{2k-n} |q|^{k(k+1)/2}\right)\\
\le\frac 32\left(\|W_{n} \|+\sum_{k=0}^{n-1} \rho^{k-n} \left(\frac 32\, \|W_{k} \|+ \left(\frac 32\right)^2\,\left[  \sum_{j=0}^{k-1} \left(\frac 52 \right)^{j}\frac{\|W_{k-1-j} \|}{\rho^{j+1}}\right]\right)\right).
\end{multline*} 

Now we sum like terms. The coefficient at $\|W_m\|$ is equal to 
$$
\left(\left(\frac 32\right)^2+
\left(\frac 32\right)^3 \left(\sum_{i=0}^{n-m-2} \left(\frac 52 \right)^{i} \right)\right)\rho^{m-n}=\left(\frac 32\right)^2 \left(\frac 52 \right)^{n-m-1}\rho^{m-n}
$$
for $m=0,\ldots,n-1$ and to $3/2$ for $m=n$. By putting $k=n-m-1$, we obtain~\eqref{fhnset}.
\end{proof}


\begin{proof}[Proof of Theorem~\ref{AMqupl}]
(A)~It suffices to show that the map $\rho$ defined in \eqref{rhodef} is topologically injective.  Moreover, since $\cO(\CC,\ell_1)\times\cO(\CC,\ell_1)$ is a direct summand in $ \cO(\CC,\B(c_0))\times \cO(\CC,\B(\ell_1))$ it suffices to prove that the map $\eta$ defined in \eqref{etadef} is topologically injective; see, e.g. \cite[Lemma 4.4]{ArNew}.

We need an explicit description of the topologies on the  target and source spaces in~\eqref{etadef}.
Note that $\cO(\Om)$ is a closed subspace of $\cO(\CC)\times \cO(\CC)$.
Writing an element of $\cO(\Om)$  as $h=(f,g)$, where $f(0)=g(0)$, we conclude that the topology on $\cO(\Om)$ is determined by the  family
$((f,g)\mapsto \max\{\|f\|_{\rho},\|g\|_{\rho}\},\,\rho>0)$, where $\|f\|_{\rho}\!:= \sup\{|f(z)|\!:\, |z|\le \rho\}$. We identify $\cO(\CC,\ell_1)\times\cO(\CC,\ell_1)$ with $\ell_1[\cO(\CC)\times\cO(\CC)]$ endowed with the following family of seminorms:
$$
\|(\bar f,\bar g)\|^\sim_{\rho}\!:=\sum_n\max \{\|f_n\|_\rho,\,\|g_n\|_\rho\}\qquad(\rho>0).
$$ 

On the other hand, recall that \eqref{Bq12} holds. 
Since $\mathfrak{B}_{|q|^{1/2}}$ is a K\"{o}the space, it follows from~\cite[Theorem 1]{Pie} that $\cO(\Om)\ptn\mathfrak{B}_{|q|^{1/2}}$ is a vector-valued K\"{o}the space of the form $\la_1[\cO(\Om)]$. This means that $\cO(\Om)\ptn\mathfrak{B}_{|q|^{1/2}}$ is topologically isomorphic to
the space of all $\cO(\Om)$-valued sequences $a=((f_n,g_n);\,n\in\Z_+)$ such that
\begin{equation}\label{rhor}
|a|_{\rho,r}\!:=\sum_{n=0}^\infty \max\{\|f_n\|_{\rho},\,\|g_n\|_{\rho}\}\,r^n |q|^{n^2/2}<\infty\qquad\text{for all $\rho,r>0$.}
\end{equation}
and the topology is determined by the family $\{|\cdot|_{\rho,r},\,\rho,r>0\}$.

To prove the topological injectivity it suffices to consider the case $r=\rho|q|^{1/2}$. Then~\eqref{rhor} takes the form
$$
\sum_{n=0}^\infty \max\{\|f_n\|_{\rho},\,\|g_n\|_{\rho}\}\,\rho^n |q|^{n(n+1)/2}.
$$
To complete the proof of Part~(A) we need to estimate this series.

Applying Lemma~\ref{hnset}, we get that it does not exceed
\begin{equation*}
\sum_{n=0}^\infty\left[\frac 32\, \max\{\|W_{n}(\bar f)\|_\rho,\,\|W_{n}(\bar g)\|_\rho\}
+ \left(\frac 32 \right)^2\left[\sum_{k=0}^{n-1} \left(\frac 52 \right)^{k}\frac{\max\{\|W_{n-1-k}(\bar f) \|_\rho,\,\|W_{n-1-k}(\bar g) \|_\rho\}}{\rho^{k+1}}\right]\right],
\end{equation*}
where $\bar f=(f_n)$ and $\bar g=(g_n)$.
When $\rho>5/2$,   the sum 
is not less than $$C\sum_n\max \{\|W_{n}(\bar f)\|_\rho,\,W_{n}(\bar g)\|_\rho\}$$ for some $C>0$. 
Hence, $|a|_{\rho,\,\rho|q|^{1/2}}\le C  \|(W_{n}(\bar f),W_{n}(\bar g))\|^\sim_{\rho}$ for every $a\in\mathcal{R}(\CC^2_q)$.
Finally, it follows from Lemma~\ref{eta12} that $\eta$ is topologically injective.


(B) Let $B$ be a Banach algebra and  $\phi\!:\mathcal{R}(\CC^2_q)\to B$ a homomorphism. Denote by $\te_x$ and $\te_y$ the global holomorphic functional calculi  (i.e., continuous homomorphisms from $\cO(\CC)$ to $B$) corresponding to $\phi(x)$ and $\phi(y)$. Recall that $\cO(\Om)$ consists of pairs $(f,g)$ of entire functions such that $f(0)=g(0)$ and consider the continuous linear map
$$
\phi_1\!:\cO(\Om)\to B\!: (f,g)\mapsto \te_x(f)+\te_y(g)-f(0).
$$

Further, since $\phi(x)\phi(y)=q\phi(y)\phi(x)$, it follows from Lemma~\ref{sile} that  $\|\phi(u)^n\|^{1/n}=O(|q|^{n/2})$ as $n\to\infty$. So by Corollary~\ref{PexunCom2},
there is a unique continuous unital homomorphism
$\phi_2\!:\mathfrak{B}_{|q|^{1/2}}\to B$ that maps $z$ to $\phi(u)$. Take the composition of the tensor product of $\phi_1$ and $\phi_2$ and the linearization of the multiplication in~$B$,
$$
\wh\phi\!:\cO(\Om)\ptn\mathfrak{B}_{|q|^{1/2}}\to B\ptn B\to B.
$$
It is easy to see that $\wh\phi\io_1=\phi$, where $\io_1$ is defined in~\eqref{iodef}. It follows from Part~(A) that $\io_1$ is a homomorphism and $\wh\phi$ is a continuous homomorphism. Since the range of $\io_1$ is dense in $\cO(\Om)\ptn\mathfrak{B}_{|q|^{1/2}}$, such $\wh\phi$ is unique. Thus, the desired universal property holds and so $\cO(\Om)\ptn\mathfrak{B}_{|q|^{1/2}}$ is the Arens--Micheal envelope of $\cO(\CC^2_q)$.
\end{proof}

\section{The envelope with respect to the class of Banach PI algebras}

Before turning to the proof of Theorem~\ref{AMquplPI}, which describes the structure of  $\cO(\CC^2_q)^{\mathrm{PI}}$, the envelope of  $\mathcal{R}(\CC^2_q)$ with respect to the class of Banach PI algebras, we recall some preliminaries about general PI algebras.

It follows from the deep Braun--Kemer--Razmyslov theorem \cite[p.\,149, Theorem 4.0.1]{KKR16} that the (Jacobson) radical of a finitely generated PI algebra over $\CC$ is nilpotent. Of course, not each Banach PI algebra is finitely generated and there exist Banach PI  (e.g., commutative) algebras with non-nilpotent radical. Nevertheless, we can apply the mentioned theorem in the quantum plane case using the following well-known result, which is a simple consequence of famous Kaplansky's theorem; see, e.g., \cite[Theorem 1.11.7]{GZ05}.


\begin{pr}\label{KaplCC}
Every irreducible representation of a PI algebra is  finite dimensional.
\end{pr}

This proposition gives a restriction on the structure of PI quotients.

\begin{lm}\label{reprone}
Every PI quotient of $\mathcal{R}(\CC^2_q)$ is commutative modulo radical.
\end{lm}
\begin{proof}
A complete list of primitive ideals of $\mathcal{R}(\CC^2_q)$ is given in \cite[p.\,136, Example II.1.2 and p.\,194, Example II.7.2]{BGo02}. The corresponding quotients are one dimensional or infinite dimensional. Therefore by Proposition~\ref{KaplCC}, a PI quotient may have only one-dimensional irreducible representations. Thus every PI quotient of $\mathcal{R}(\CC^2_q)$  is commutative modulo radical.
\end{proof}


The following proposition is an analogue of Lemma~\ref{sile} but we do not even assume that $X$ and $Y$ are elements of a Banach algebra.

\begin{pr}\label{XYnil}
Let $X$ and $Y$ be elements of a PI algebra such that $XY=qYX$ for some $q\ne 1$. Then $XY$ is nilpotent.
\end{pr}
\begin{proof}
Let $A$ be the subalgebra generated by $X$ and $Y$. Obviously,  $A$ also satisfies a PI. Being a quotient of $\mathcal{R}(\CC^2_q)$, the algebra $A$ is commutative modulo radical by Lemma~\ref{reprone}. Since $XY=(1-q^{-1})^{-1}[X,Y]$, we have that  $XY$ belongs to the commutant and hence the radical. As mentioned above, the Braun--Kemer--Razmyslov theorem \cite[p.\,149, Theorem 4.0.1]{KKR16} implies that the radical of a finitely generated PI algebra over $\CC$ is nilpotent. In particular, $XY$ is nilpotent.
\end{proof}

We need a pair of results on locally BPI algebras.

\begin{pr}\label{subapr}
Every closed subalgebra of a product of locally BPI algebras is a locally BPI algebra.
\end{pr}
\begin{proof}
First, using the same argument as for the Part~(B) of Proposition 2.11 in \cite{ArOld} we conclude that
an Arens--Michael algebra is a locally BPI algebra if and only if it is topologically isomorphic to a closed subalgebra of a product of Banach PI algebras. Reasoning in the same way as in the proof of Theorem~2.7 in \cite{ArNew}, one can deduce that a closed subalgebra of a product of locally BPI algebras is also a locally BPI algebra.
\end{proof}
\begin{pr}\label{OCAPI}
If $B$ is a finite-dimensional Banach algebra, then $\cO(\CC,B)$ is a locally BPI algebra.
\end{pr}
\begin{proof}
First note that $\cO(\CC,B)\cong \cO(\CC)\ptn B$. Denote by $A(\overline{\DD}_\rho)$ the Banach algebra of continuous functions on the closed disc of radius $\rho$ having holomorphic restriction to its interior. Since  $\cO(\CC)$ is a projective limit of the Banach algebras system $(A(\overline{\DD}_\rho);\,\rho\to \infty)$, we have that $\cO(\CC)\ptn B$ is
a projective limit of the system  $(A(\overline{\DD}_\rho)\ptn B;\,\rho\to \infty)$. (Projective limits commutes with $\ptn$.) To complete the proof it suffices to show that each of $A(\overline{\DD}_\rho)\ptn B$ is a PI algebra.

Note that both $A(\overline{\DD}_\rho)$ and $B$ are PI algebras; the former being commutative and the latter being finite dimensional \cite[p.\,21, Corollary 1.2.25]{KKR16}. Therefore by Regev's theorem \cite[p.\,138, Theorem 3.4.7]{KKR16},  $A(\overline{\DD}_\rho)\otimes B$ also satisfies a PI. Finally, since $B$ is finite dimensional, $A(\overline{\DD}_\rho)\ptn B\cong A(\overline{\DD}_\rho)\otimes B$ as an associative algebra.
\end{proof}

We also need the following trivial lemma.
\begin{lm}\label{forseup}
Let $b$ be a nilpotent element of a Banach algebra~$B$. Then there is a unique continuous unital homomorphism $\psi\!:\CC[[z]]\to B$ that maps $z$ to $b$.
\end{lm}

\begin{proof}[Proof of Theorem~\ref{AMquplPI}]
The argument is very similar to that for Theorem~3.8 in \cite{ArNew}. The difference is
that we replace everywhere $\R$ by $\CC$ and algebras of real-valued functions of class $C^\infty$ by  algebras of holomorphic functions.

(A)~It follows from Proposition~\ref{subapr} that it suffices to construct a homomorphism from $\cR(\CC^2_q)$ to a product of locally BPI algebras and extend it to a topologically injective continuous linear map defined on $\cO(\Om)\ptn\CC[[u]]$.

Let $\rT_p$ denote the algebra of upper triangular (complex) matrices of order $p$. Proposition~\ref{OCAPI} implies that $\cO(\CC,\rT_p)$ is a locally BPI algebra. Using the same construction as in \cite[Theorem~4.3]{ArNew} we obtain a continuous linear map
$$
\rho\!:\cO(\Om)\ptn \CC[[u]]\to \prod_p \cO(\CC,\rT_p)^2
$$
whose restriction to $\cR(\CC^2_q)$ is a homomorphism.  The proof of the topologically injectivity of $\rho$ is also the same  with the only difference that instead of the fact that an ideal in an algebra of type $C^\infty$ generated by a polynomial is closed we use the fact that a similar ideal in an algebra of holomorphic functions is closed; see, e.g., \cite[\S\,V.6.4, p.\,169, Corollary~2]{GR79}.

The proof of Part~(B) is analogous to that of Part~(B) of Theorem~\ref{AMqupl}. The difference is that we assume that $B$ is a PI algebra and use Proposition~\ref{XYnil} (which implies that $\phi(u)$ is nilpotent) instead of Lemma~\ref{sile} and the universal property of $\CC[[u]]$ for nilpotent elements in Banach algebras in Lemma~\ref{forseup}.
\end{proof}

\begin{rem}
Suppose as usual that $xy=qyx$, where $|q|<1$. In \cite[\S\,5.5]{Do24b} Dosi introduced the following algebra (with notation $\Ga(\mathcal{F}_q,\CC^2_{xy})$ and $q^{-1}$ instead of $q$):
$$
\mathcal{F}(\CC^2_q)\!:=\Bigl\{a=\sum_{k,l=0}^\infty
\al_{kl} y^k x^l\!:
 \|a\|'_{r,l}<\infty \;\;\forall r>0,\,l\in\Z_+,\quad
 \|a\|''_{r,k}<\infty \;\;\forall r>0,\,k\in\Z_+\Bigr\},
$$
where
$$
\|a\|'_{r,l}\!:=\sum_{k=0}^\infty |\al_{kl}| r^{k}\quad(r>0,\,l\in\Z_+)\quad\text{and}\quad
\|a\|''_{r,k}\!:=\sum_{l=0}^\infty |\al_{kl}| r^{l}\quad(r>0,\,k\in\Z_+).
$$

We claim that the linear isomorphism $\cR(\CC^2_q)\to\cR(\Om)\otimes \CC[u]$ defined in~\eqref{Riso} extends to a topological isomorphism $\mathcal{F}(\CC^2_q)\to \cO(\CC^2_q)^{\mathrm{PI}}$. It suffices to check the coincidence of the topologies on $\cR(\CC^2_q)$ and $\cR(\Om)\otimes \CC[u]$.

Every $a$ in $\cR(\CC^2_q)$ can be written in the following form:
$$
a=\sum_j\Bigl(\sum_{i\ge 0}\be_{ij}u^jx^i+\sum_{i>0}\ga_{ij}y^iu^j\Bigr).
$$
Put also $\ga_{0j}=\be_{0j}$. Theorem~\ref{AMquplPI} implies that $\cO(\CC^2_q)^{\mathrm{PI}}\cong\cO(\Om)\ptn\CC[[u]]$. It is easy to see that the
topology on the latter space can be determined by the following family of norms: 
$$|a|'_{r,k}= \sum_i|\be_{ik}|r^i, \quad|a|''_{r,l}=\sum|\ga_{il}|r^i \quad(r>0,\,k,l\in\Z_+).$$

Since $(xy)^j=q^{j(j+1)/2}y^jx^j$ 
, we have that
$$
a=\sum \be_{ij}\,q^{j(j+1)/2}y^{j}x^{i+j}+ \sum\ga_{ij}\,q^{j(j+1)/2}y^{i+j}x^i.
$$
Hence,
$$
\|a\|'_{r,l}=\sum_{k\ge l} |\be_{k-l,l}\,q^{l(l+1)/2}|\,r^k+ \sum_{k=1}^{l}|\ga_{k,l-k}\,q^{(l-k)(l-k+1)/2}|\,r^k
$$
and
$$
\|a\|''_{r,k}=\sum_{l\ge k} |\be_{l-k,k}\,q^{k(k+1)/2}|\,r^l+ \sum_{l=1}^{k}|\ga_{l,k-l}\,q^{(k-l)(k-l+1)/2}|\,r^l.
$$

These formulas  obviously imply that $\|\cdot\|'_{r,l} $ and  $\|\cdot\|''_{r,k} $ are majorized by the family $\{|\cdot|'_{r,k},\,|\cdot|''_{r,l}\}$ (since $|q|<1$). An estimate in the reverse direction follows from the facts that the second sums in the above formulas are finite and $q\ne0$.
\end{rem}

\end{document}